\documentclass{amsart}
\usepackage{amsmath,amssymb,amsthm}
\usepackage{tikz}
\usetikzlibrary{calc}
\usepackage{hyperref}

\hypersetup{
    colorlinks=false,
    pdfborder={0 0 1},
    linkbordercolor={1 0 0},
    citebordercolor={0 1 0},
    urlbordercolor={0 1 0},
    filebordercolor={0 1 0}
}
\numberwithin{equation}{section}
\newtheorem{theorem}{Theorem}[section]
\newtheorem{lemma}[theorem]{Lemma}
\newtheorem{definition}[theorem]{Definition}

\newcommand{\orcid}[1]{\href{https://orcid.org/#1}{\texttt{ORCID: #1}}}
\newcommand{\ka}{\kappa}
\newcommand{\Om}{\Omega}
\newcommand{\cG}{\mathcal {G}}
\newcommand{\bB}{{\bf B}}
\newcommand{\bT}{{\bf T}}
\newcommand{\Kc}{K_\circ}
\newcommand{\cB}{{\mathcal B}}
\newcommand{\cQ}{{\mathcal Q}}
\newcommand{\B}{\mathbb{B}}
\newcommand{\supp}{{\rm supp}}

\begin{document}

\title[Wave Packet Density in Local Smoothing]{Wave Packet Density in Local Smoothing: Refinements and Limitations}

\author{Xiangyu Wang} \address{Xiangyu Wang \\ Department of Mathematics University of Illinois Urbana-Champaign\newline \orcid{0009-0003-5983-6961}} \email{xw70@illinois.edu}

\begin{abstract}
We develop a transverse covering $L^2$ argument that both substantially simplifies the approach of Gan, He, Li, and Wu and improves their local smoothing estimates for Fourier integral operators with positive conic curvature in even spatial dimensions $n\geq 4$. We also investigate the limitations of wave packet density methods, showing that genuine geometric obstructions already arise in the Euclidean setting.
\end{abstract}


\maketitle

\section{Introduction}
\label{sec:introduction}

Local smoothing estimates quantify the gain in regularity that is obtained when a solution of the wave equation is averaged in time. Let
\[
u(x,t)=e^{it\sqrt{-\Delta}}f(x)
\]
denote the half-wave evolution on $\mathbb{R}^{n}$. At each fixed time, the classical estimates of Peral and of Seeger, Sogge, and Stein show that the half-wave propagator loses
\[
s_p=(n-1)\left|\frac12-\frac1p\right|
\]
derivatives in $L^p$ \cite{Peral1980,SeegerSoggeStein1991}. Sogge observed that averaging over a compact time interval yields additional regularity, and he formulated the local smoothing conjecture \cite{Sogge1991}. For $p>2$, the conjecture asserts that, for every $\varepsilon>0$,
\[
\|e^{it\sqrt{-\Delta}}f\|_{L^p(\mathbb{R}^n\times[1,2])}
\lesssim_{\varepsilon}
\|f\|_{L^p_{s_p-\frac1p+\varepsilon}(\mathbb{R}^n)}
\]
whenever
\[
p\geq \frac{2n}{n-1}.
\]
Thus averaging in time is expected to gain almost $1/p$ derivatives over the fixed-time estimate.

The local smoothing problem is closely related to Fourier restriction, to Kakeya-type problems, and to oscillatory integral estimates. Wolff introduced decoupling inequalities for the cone as a tool for this problem \cite{Wolff2000}. The sharp $\ell^2$ decoupling theorem of Bourgain and Demeter \cite{BourgainDemeter2015} implies the local smoothing conjecture in the range
\[
p>2+\frac{4}{n-1},
\]
and Beltran, Hickman, and Sogge \cite{BeltranHickmanSogge2020} proved the corresponding estimates for Fourier integral operators satisfying the cinematic curvature condition. In two spatial dimensions, Guth, Wang, and Zhang proved the full conjecture for the Euclidean wave equation by means of a sharp square function estimate for the cone \cite{GuthWangZhang2020}; the corresponding result for Fourier integral operators and for wave equations on compact Riemannian surfaces was then obtained by Gao, Liu, Miao, and Xi \cite{GaoLiuMiaoXi2023SF}.

In dimensions $n\geq3$, Gan, He, Li, and Wu \cite{GanHeLiWu} recently went beyond the decoupling exponent. They introduced the \emph{wave packet density}, a quantity that is compatible with Lorentz rescaling, and combined it with refined decoupling, broad estimates, and a two-ends reduction. Their results apply to wave equations on compact Riemannian manifolds. In that generality one cannot expect the full range $p\geq 2n/(n-1)$: Minicozzi and Sogge \cite{MinicozziSogge1997} constructed examples which show that, on a general compact Riemannian manifold, the local smoothing estimate fails for $p<p_n^{+}$, where
\[
p_n^{+}
:=
\begin{cases}
\displaystyle 2+\frac{8}{3n-3}, & n \text{ odd},\\[6pt]
\displaystyle 2+\frac{8}{3n-2}, & n \text{ even}.
\end{cases}
\]
Gan, He, Li, and Wu proved the estimate for all $p>p_n^{+}$ when $n$ is odd, which is sharp in view of these examples. When $n$ is even, they proved it for
\[
p > 2+\frac{8}{3n-2-\omega(n)},\qquad \omega(n) = \frac{4(7n-2)}{3n^2+5n+14},
\]
an exponent that is within $O(n^{-3})$ of $p_n^{+}$.

OpenAI has recently announced a proof of Sogge's local smoothing conjecture for the Euclidean wave equation in every spatial dimension $n\geq2$ \cite{OpenAI2026LocalSmoothing}. The announced result is specific to the Euclidean setting, whereas the local smoothing estimates proved in this paper apply to Fourier integral operators with variable coefficients and positive conic curvature. Our arguments do not depend on it. We return to \cite{OpenAI2026LocalSmoothing} in Section~\ref{subsec:intro-limitations}, where we discuss the second part of the paper.

\subsection{An improved range in even dimensions}\label{subsec:intro-refinement}
The method of \cite{GanHeLiWu} rests on an estimate that bounds the $L^p$ norm of the solution on a ball of radius $R$ by a combination of the $L^2$ norm of the data and its wave packet density; see \eqref{wave-packet-density-estimate} below. We call it the \emph{wave packet density estimate}. Our first result is this estimate for Fourier integral operators with positive conic curvature in every even dimension $n\geq4$, for all
\[
p\geq p_*(n):=
2+\frac{8}{3n-2-\frac4{n+1}};
\]
see Theorem~\ref{thm:refined-GHLW-density}. Consequently, the local smoothing estimates for wave equations on compact Riemannian manifolds of even dimension $n\geq4$ hold for $p>p_*(n)$. Compared with \cite{GanHeLiWu}, the term $\omega(n)$ is replaced by $4/(n+1)$. For instance, $p_*(4)=66/23\approx2.8696$, whereas the exponent of \cite{GanHeLiWu} for $n=4$ is $522/179\approx2.9162$.

The proof follows the scheme of \cite{GanHeLiWu}. After a two-ends reduction and a broad--narrow reduction, it remains to estimate a broad norm of a sum of wave packets on a union $X$ of balls of radius $K^2$, where $K$ is a small power of $R$. This is done by interpolating between three estimates: refined decoupling at the exponent $2(n+1)/(n-1)$, a broad estimate of Schippa \cite{Schippa2024}, and an $L^2$ estimate. Our contribution is a new $L^2$ estimate, Lemma~\ref{lem:transverse-covering}. The two-ends reduction provides a number $\mu\geq1$ with the following property: every ball of $X$ meets at least $\mu$ wave packets, even after the wave packets with directions in any $A$ given $K^{-1}$-caps have been discarded. Here a $K^{-1}$-cap is a cap of angular width $K^{-1}$ in the frequency sector, and $A$ is a large parameter. Hence, for each $K^{-1}$-cap $\tau$, the set $X$ is covered at least $\mu$ times by the planks of wave packets whose directions are transverse to $\tau$. On each such plank, the part of the function with frequencies in $\tau$ satisfies a favorable $L^2$ bound \cite[Lemma~3.17]{GanHeLiWu}. Summing these bounds over the covering planks, whose number is controlled by an incidence estimate of hairbrush type from \cite{GanHeLiWu} (Lemma~\ref{hairbrush}), gives an $L^2$ estimate with a gain of $\mu^{-1}$. The same number $\mu$ also bounds from above the multiplicity of the wave packets in the refined decoupling estimate, where it appears as a loss; the gain and the loss cancel in the interpolation. The argument uses only the numbers of wave packets before and after each step of the two-ends reduction; it does not need the regularity parameters that enter the $L^2$ estimate of \cite{GanHeLiWu}.

\subsection{Limitations of wave packet density}\label{subsec:intro-limitations}
Our second result concerns the Euclidean wave equation. For general operators of the class considered above, the wave packet density estimate cannot hold for $p<p_n^{+}$, because it would contradict the examples of Minicozzi and Sogge. In Euclidean space the conjectured range of local smoothing is larger, and one may ask whether the wave packet density estimate holds there for some $p<p_n^{+}$. We show that it does not. For every $n\geq3$ we construct a family of wave packets, tangent to a hyperboloid, for which the estimate fails for all $2n/(n-1)<p<p_n^{+}$ (Theorem~\ref{thm:counterexample}). These wave packets overlap heavily near the hyperboloid, but the wave packet density does not detect this: the density of the family is comparable to that of a single wave packet. We also examine the anisotropic multiscale version of wave packet density proposed in \cite[Remark~4.9]{GanHeLiWu}. It takes more of the geometry into account and weakens the obstruction, but it does not remove it: for $n=3$, the same family violates the corresponding estimate for all $3<p<22/7$ (Theorem~\ref{thm:counter-example-multiscale-bound}), whereas $p_3^{+}=10/3$.

These examples already appeared in an earlier version of this paper, which was made public before \cite{OpenAI2026LocalSmoothing}. Now that a proof of the Euclidean conjecture has been announced, they can no longer serve as a guide towards its solution. We believe that they are still of interest, because they explain why the approach through wave packet density, which gives the sharp range on manifolds of odd dimension, cannot give the conjectured range in Euclidean space in its present form: the estimate that this approach is designed to prove is false below $p_n^{+}$, even for the Euclidean wave equation.

\subsection*{Organization of the paper}
In Section~\ref{sec:preliminaries} we recall the wave packet density, state our results, and collect the tools from \cite{GanHeLiWu} that we need. Section~\ref{sec:refinement} contains the proof of Theorem~\ref{thm:refined-GHLW-density}. Section~\ref{sec:counterexample} begins with a heuristic explanation of how the wave packet density estimate can fail, and then proves Theorems~\ref{thm:counterexample} and~\ref{thm:counter-example-multiscale-bound}.

\subsection*{Notation and parameters}
Throughout this paper we use the notation, definitions, and standing assumptions of \cite{GanHeLiWu}. In particular, $\Theta_s$ denotes the family of $s$-caps for $R^{-1/2}\leq s\leq1$; we call the elements of $\Theta_{K^{-1}}$ the $K^{-1}$-caps and those of $\Theta_{R^{-1/2}}$ the $R^{-1/2}$-caps, where $K$ is the small power of $R$ fixed in Section~\ref{subsec:broadnorm}. We write $A\lessapprox B$ if $A\lesssim R^{\eta}B$ for an exponent $\eta>0$ that can be chosen arbitrarily small. In Section~\ref{sec:refinement} these exponents are chosen so small that all such losses are bounded by $R^{O(\varepsilon^2)}$. The parameters satisfy $0<\delta\ll\varepsilon^{500}\ll1$. As usual, we neglect the rapidly decaying tails of wave packets; accordingly, all local estimates are to be understood with fixed smooth cutoffs and slightly enlarged sets.

\subsection*{AI Declaration}
In an early version of this paper, the author combined a $k$-broad incidence argument, a quadratic cone condition, and a standard $L^2$ estimate within the framework of Gan, He, Li, and Wu; this reduced their correction term $\omega(n)$ to $6/n+O(n^{-2})$. The author then observed that applying the hairbrush bound \eqref{eq:hairbrush-cancellation} yields a sharper $L^2$ estimate, which improved the correction term to $4/n+O(n^{-2})$. Starting from the author's approach, ChatGPT simplified and strengthened this argument into the transverse covering argument presented here, which reduces the correction term further to $4/(n+1)$. ChatGPT also found that, once the transverse covering estimate is used, the $k$-broad incidence argument, the quadratic cone condition, and the standard $L^2$ estimate no longer improve the final exponent. They have therefore been removed from the paper.

In addition, ChatGPT was used to assist with the interpolation calculations in Section~\ref{sec:interpolation} and to identify the optimal choices of the parameters $l$ and $\beta$ in the example of Section~\ref{sec:counterexample}. Claude was used to improve the English exposition.

\section{Preliminaries and statements of the results}\label{sec:preliminaries}
\subsection{Wave packet density}\label{subsec:wavepacketdensity}
The notion of wave packet density has its origin in the wave envelope estimates of Guth, Wang, and Zhang \cite{GuthWangZhang2020}, in which the concentration of wave packets is measured by localized square functions at several scales. Gan, He, Li, and Wu \cite{GanHeLiWu} turned this idea into a single quantity that is defined in terms of $L^2$ norms and is compatible with Lorentz rescaling, and they made it the basis of their approach to local smoothing in higher dimensions.

Following \cite{GanHeLiWu}, we use wave packet density as a substitute for $\|f\|_{\infty}$. Instead of the pointwise size of $f$, it measures how strongly the $L^2$ mass of the wave packets of $f$ is concentrated. It is defined through sums of squared $L^2$ norms, which are nonnegative; therefore it behaves well under almost orthogonality and under the passage to subfamilies of wave packets, and this makes it convenient for induction on scales.

Let $R^{-1/2}\leq s\leq 1$, let $\tau\in\Theta_s$ be an $s$-cap, and let $\tau^*$ denote the box dual to $\tau$. Put $V_{\tau,R}:=(Rs^2)\tau^*$; this is a box centered at the origin with dimensions $Rs^2\times Rs\times\cdots\times Rs$. We write $V\parallel V_{\tau,R}$ if $V$ is a translate of $V_{\tau,R}$. As in \cite{GanHeLiWu}, for $\theta\in\Theta_{R^{-1/2}}$ we denote by $\mathbb{T}_{\theta}(z_0)$ the family of wave packet planks at scale $R$ in $B_R^{n+1}(z_0)$ that are associated with $\theta$, and by $T^{\flat}$ the base box associated with a plank $T$, which is a box of dimensions $1\times R^{1/2}\times\cdots\times R^{1/2}$ in $\mathbb R^n$. For a function $g$, we write $g_T$ for the wave packet of $g$ associated with $T$.

We recall the following definition from
\cite[Definition~4.1]{GanHeLiWu}.

\begin{definition}[Wave packet density]
\label{def:wave-packet-density}
Let $g\in\mathcal{S}(\mathbb{R}^n)$ satisfy
$\supp\widehat{g}\subset B_1^n \setminus B_{1/2}^n$, and let
\[
    B_R^{n+1}(z_0)
    =
    B_R^{n+1}(x_0,t_0)
    \subset B_{\lambda}^{n+1}.
\]
The wave packet density of $g$ in $B_R^{n+1}(z_0)$ is defined by
\[
W(g,B_R^{n+1}(z_0)) :=
    \sup_{R^{-1/2}\leq s\leq 1}
    \sup_{\tau\in\Theta_s}
    \sup_{V\parallel V_{\tau,R}}
    \left(
        \frac{1}{|V|}
        \sum_{\substack{
            \theta\in\Theta_{R^{-1/2}}\\
            \theta\subset\tau
        }}
        \sum_{\substack{
            T\in\mathbb{T}_{\theta}(z_0)\\
            T^{\flat}\subset V
        }}
        \|g_T\|_{L^2(\mathbb{R}^n)}^2
    \right)^{1/2}.
\]
\end{definition}

We shall use the following elementary lower bounds for the wave packet density in Section~\ref{sec:interpolation}.

\begin{lemma}[{\cite[Lemma~4.4]{GanHeLiWu}}]
\label{lem:density-lower-bounds}
Let $g=\sum_{T\in\mathbb T}g_T$ be a nonempty sum of wave packets
at scale $R$ in $B_R^{n+1}$. Suppose that
\[
\|g_T\|_2^2\gtrsim R^{(n-1)/2}
\qquad(T\in\mathbb T).
\]
Then
\[
W(g,B_R^{n+1})^2
\gtrsim
\max\left\{1,\frac{\#\mathbb T}{R^{(n+1)/2}}\right\}.
\]
\end{lemma}

The first lower bound is obtained by taking $V=T^\flat$ for a single $T\in\mathbb T$ in the definition of $W$, and the second by taking $s=1$. Since all terms in the definition of $W$ are nonnegative, both bounds remain valid when $\mathbb T$ is only a nonempty subfamily of the wave packets of $g$ whose members satisfy the same normalization.

\subsection{Statements of the results}\label{subsec:results}
Let $\mathcal F$ be a Fourier integral operator of type~$\mathbf1$ whose phase and amplitude are normalized as in \cite[Sections~3.1--3.2]{GanHeLiWu}. In particular, the conic hypersurfaces associated with the phase have $n-1$ nonzero principal curvatures, which are all positive for a suitable choice of the normal. We write $\mathcal F^\lambda$ for the rescaled operator. We work with functions $f$ whose Fourier transforms are supported in a fixed sector:
\begin{equation}\label{supp-f-hat}
\begin{aligned}
\supp\widehat f\subset\mathbb A^n(1):=
\bigl\{\xi\in B_1^n\setminus B_{1/2}^n:
&\ \xi_j>0\text{ for }1\leq j\leq n-1,\\
&\ |(\xi_1,\ldots,\xi_{n-1})|\leq\xi_n\bigr\}.
\end{aligned}
\end{equation}

\begin{theorem}\label{thm:refined-GHLW-density}
Let $n\geq4$ be even, and let $\mathcal F$ be a normalized
type-$\mathbf1$ Fourier integral operator as above. For every
$\varepsilon>0$ and every $p\geq p_*(n)$, one has
\begin{equation}\label{wave-packet-density-estimate}
\|\mathcal F^\lambda f\|_{L^p(B_R^{n+1})}
\lesssim_{\varepsilon}
R^{(n-1)(\frac12-\frac1p)+\varepsilon}
\|f\|_2^{2/p}W(f,B_R^{n+1})^{1-2/p}
\end{equation}
for all $1\leq R\leq\lambda^{1-\varepsilon}$ and all
$f\in\mathcal S(\mathbb R^n)$ satisfying \eqref{supp-f-hat}.
The implicit constant is uniform over the class of normalized phases and amplitudes.
\end{theorem}

We refer to \eqref{wave-packet-density-estimate} as the \emph{wave packet density estimate} with exponent $p$. By the reduction in \cite[Sections~3.1 and 4]{GanHeLiWu}, Theorem~\ref{thm:refined-GHLW-density} implies the usual local smoothing estimate for wave equations on compact Riemannian manifolds of even dimension $n\geq4$, for all $p>p_*(n)$.

The wave packet density estimate is stronger than the corresponding local smoothing estimate. The density $W$ only records the $L^2$ mass of those wave packets whose directions lie in a cap $\tau$ and whose bases lie in a box $V\parallel V_{\tau,R}$; it does not detect every way in which wave packets can overlap in spacetime. Our next result exhibits wave packets, tangent to a hyperboloid, which overlap heavily although their wave packet density is comparable to that of a single wave packet. It shows that, even in the Euclidean setting, the wave packet density estimate fails below the exponent $p_n^{+}$ of Minicozzi and Sogge.

\begin{theorem}\label{thm:counterexample}
Let $n\geq 3$. For every sufficiently large $R$ there is a function $f_R$ with
\[
\operatorname{supp}\widehat{f_R}\subset B_1^n \setminus B_{1/2}^n
\]
such that, for every $\frac{2n}{n-1} < p < p_n^{+}$,
\[
\left\|e^{it\sqrt{-\Delta}}f_R\right\|_{L^p(B_R^{n+1})}
\gtrsim
R^{(n-1)\left(\frac12-\frac1p\right)+c_{n,p}}
\|f_R\|_2^{\frac{2}{p}}
W(f_R,B_R^{n+1})^{1-\frac{2}{p}},
\]
where the constant $c_{n,p}>0$ and the implicit constant do not depend on $R$.
\end{theorem}

One may ask whether a finer notion of density avoids this obstruction. We consider the anisotropic multiscale version of wave packet density
suggested in \cite[Remark 4.9]{GanHeLiWu}. Let
\[
R^{-1/2}\leq s_1,\ldots,s_{n-1}\leq 1
\]
be dyadic numbers, and set $s_{\max}:=\max_{1\leq i\leq n-1}s_i$.
Let $\tau\subset B_1^n \setminus B_{1/2}^n$ be an anisotropic cap of dimensions $1\times s_1\times\cdots\times s_{n-1}$; we call it an $(s_1,\ldots,s_{n-1})$-cap. We define $V_{\tau,R}$ to be the box centered at the origin with dimensions $Rs_{\max}^2 \times Rs_1 \times\cdots\times Rs_{n-1}$ whose edges are parallel to the corresponding edges of $\tau$; thus the edge of length $Rs_{\max}^2$ points in the radial direction of $\tau$.

Suppose that $g\in\mathcal{S}(\mathbb{R}^n)$ satisfies $\operatorname{supp}\widehat g\subset B_1^n \setminus B_{1/2}^n$. The \emph{multiscale wave packet density} of $g$ in
$B_R^{n+1}(z_0)$ is defined by
\[
\widetilde{W}(g,B_R^{n+1}(z_0))
:=
\sup_{\substack{
R^{-1/2}\leq s_1,\ldots,s_{n-1}\leq1\\
s_1,\ldots,s_{n-1}\ {\rm dyadic}
}}
\!
\sup_{\tau}
\!
\sup_{\substack{
V\parallel V_{\tau,R}\\
V\subset B_R^n
}}
\left(
\frac{1}{|V|}
\sum_{\substack{
\theta\in\Theta_{R^{-1/2}}\\
\theta\subset\tau
}}
\sum_{\substack{
T\in\mathbb{T}_{\theta}(z_0)\\
T^\flat\subset V
}}
\|g_T\|_2^2
\right)^{1/2},
\]
where the second supremum is taken over all
$(s_1,\ldots,s_{n-1})$-caps $\tau$.

The multiscale density takes more of the geometry into account. Nevertheless, in three spatial dimensions the functions of Theorem~\ref{thm:counterexample} still violate the corresponding estimate for some exponents.

\begin{theorem}\label{thm:counter-example-multiscale-bound}
Let $n=3$, and let $f_R$, $R\gg1$, be the functions constructed in the proof of
Theorem~\ref{thm:counterexample}. Then, for every
\[
3<p<\frac{22}{7},
\]
we have
\[
\left\|e^{it\sqrt{-\Delta}}f_R\right\|_{L^p(B_R^4)}
\gtrsim
R^{2\left(\frac12-\frac1p\right)+c_p}
\|f_R\|_2^{\frac{2}{p}}
\widetilde{W}(f_R,B_R^4)^{1-\frac{2}{p}},
\]
where the constant $c_p>0$ and the implicit constant do not depend on $R$.
\end{theorem}
Since $p_3^{+}=10/3>22/7$, the range of exponents in Theorem~\ref{thm:counter-example-multiscale-bound} is smaller than the range $3<p<10/3$ that Theorem~\ref{thm:counterexample} gives for $n=3$. In this sense the multiscale density weakens the obstruction, but it does not remove it.

\subsection{Broad norms and the broad--narrow reduction}\label{subsec:broadnorm}
We recall the broad--narrow method of Bourgain and Guth
\cite{BourgainGuth} in the form used in \cite{GanHeLiWu}.
Let $K=R^{\varepsilon^{50}}$. We cover the frequency sector \eqref{supp-f-hat} by the $K^{-1}$-caps $\tau\in\Theta_{K^{-1}}$, which have bounded overlap, and we choose a smooth partition of unity
$\{\psi_\tau\}$ such that each $\psi_\tau$ is supported in a fixed interior subcap of $\tau$. We write
\[
P_\tau f=(\psi_\tau\widehat f)^\vee,
\qquad f_\tau=P_\tau f.
\]
When necessary, the caps are enlarged by fixed factors, so that
$P_\tau f_T\neq0$ implies $\theta(T)\subset\tau$ for every wave packet $f_T$ at scale $R$ under consideration.

Fix an integer $1\leq k\leq n+1$ and a sufficiently large integer $A$, and let $G^\lambda$ denote the rescaled generalized Gauss map.
For a $K^2$-ball $B\subset B_\lambda^{n+1}$ centered at $z_B$, set
\[
\mu_{\mathcal F^\lambda f}(B)
:=
\min_{V_1,\ldots,V_A}
\max_{\tau\notin V_i}
\int_B|\mathcal F^\lambda f_\tau|^p.
\]
Here the minimum is taken over all $A$-tuples of $(k-1)$-dimensional subspaces $V_1,\ldots,V_A$ of $\mathbb R^{n+1}$, and the maximum over all $K^{-1}$-caps $\tau$ such that
\[
\angle\bigl(G^\lambda(z_B;\tau),V_i\bigr)>K^{-2}
\qquad(1\leq i\leq A);
\]
this condition is abbreviated by $\tau\notin V_i$. The angle between a set of directions and a subspace is the infimum of the angles between its elements and the subspace, and the maximum over an empty family is zero.
For a disjoint union $U$ of $K^2$-balls, define
\[
\|\mathcal F^\lambda f\|_{BL_{k,A}^p(U)}^p
:=\sum_{B\subset U}\mu_{\mathcal F^\lambda f}(B).
\]
We call $A$ the broadness parameter.

We shall use the H\"older inequality for broad norms introduced by Guth
\cite[Lemma~4.2]{Guth2018}, in the version for Fourier integral operators given in
\cite[Lemma~3.16]{GanHeLiWu}.
If $1\leq p,p_1,p_2<\infty$, $\alpha_1,\alpha_2\geq0$, and
\[
\alpha_1+\alpha_2=1,\qquad
\frac1p=\frac{\alpha_1}{p_1}+\frac{\alpha_2}{p_2},
\]
then, for positive integers $A_1,A_2$,
\begin{equation}\label{eq:broad-holder}
\|\mathcal F^\lambda f\|_{BL_{k,A_1+A_2}^{p}(U)}
\lesssim
\|\mathcal F^\lambda f\|_{BL_{k,A_1}^{p_1}(U)}^{\alpha_1}
\|\mathcal F^\lambda f\|_{BL_{k,A_2}^{p_2}(U)}^{\alpha_2}.
\end{equation}
Applying this inequality twice gives the corresponding inequality with three factors, in which the three broadness parameters are added.

By the rescaling and induction argument of \cite[Section 8]{GanHeLiWu}, the estimate \eqref{wave-packet-density-estimate} follows from the corresponding estimate for the broad norm, provided that
\[
p\geq2+\frac4{2n-\max\{3,k\}+1}
\]
and that the broad estimate holds uniformly over the class of normalized phases and over the unions of $K^2$-balls that are selected in the argument.
This reduction is not affected if the broadness parameter is multiplied by a fixed constant depending only on the dimension.

\subsection{A transverse \texorpdfstring{$L^2$}{L2} estimate}
\label{subsec:transverse-l2}
The following estimate bounds a sum of wave packets on a thin slab that is transverse to their Gauss directions. It is the analytic input for the covering argument of
Section~\ref{sec:l2-estimate}.

\begin{lemma}[{\cite[Lemma~3.17]{GanHeLiWu}}]
\label{lem:GHLW-transverse-l2}
Suppose $1\leq K^{1000}\leq R$. Let $S\subset\mathbb R^{n+1}$ be
a slab of dimensions $1\times R^{1/2}\times\cdots\times R^{1/2}$,
parallel to an $n$-dimensional subspace $V$. Let
$\tau\in\Theta_{K^{-1}}$, and suppose that
\[
\angle\bigl(G^\lambda(z_0;\tau),V\bigr)>K^{-2}
\]
for some $z_0\in S$. Then, for every family $\mathbf T$ of distinct
$R$-planks with $\theta(T)\subset\tau$ for all $T\in\mathbf T$,
one has
\[
\left\|\sum_{T\in\mathbf T}\mathcal F^\lambda f_T
\right\|_{L^2(S)}^2
\lesssim K^{O(1)}
\sum_{T\in\mathbf T}\|f_T\|_2^2.
\]
\end{lemma}

The condition $K^{1000}\leq R$ holds because $K=R^{\varepsilon^{50}}$ and $\varepsilon$ is small. We shall use the lemma in two slightly more general situations. First, we apply it to the functions $P_\tau f_T$ in place of $f_T$. This is permitted because, up to fixed enlargements of the caps and the planks, the cutoff $P_\tau$ preserves the localization of $f_T$ to its $R^{-1/2}$-cap in frequency and to its plank in space. Second, we apply it to slabs of thickness $K^{O(1)}$ rather than $1$. Such a slab is covered by $K^{O(1)}$ slabs of thickness $1$, and the lemma applies to each of them, provided that the transversality condition holds on each of them.

\section{Refinement: proof of Theorem~\ref{thm:refined-GHLW-density}}\label{sec:refinement}
Throughout this section, $n\geq4$ is even and
\[
k=\frac{n+4}{2},\qquad p=p_*(n).
\]
The operators and the parameters are those of Section~\ref{sec:preliminaries}. By the pigeonholing of wave packets and the homogeneity argument in \cite[Section 7]{GanHeLiWu},
it suffices to consider a family $\mathbb T^{(0)}$ of wave packets of $f$ for which
\[
\|f_T\|_2^2\sim R^{\frac{n-1}{2}}
\qquad(T\in\mathbb T^{(0)}).
\]
The function $f$ and its wave packets are fixed throughout the
argument; each of the refinements below selects a subfamily of these wave packets.

\subsection{A simplified two-ends reduction}\label{sec:two-end-reduction}
We use the two-ends algorithm of \cite[Sections~6--8]{GanHeLiWu}, but we keep only those parts of it that are needed for the transverse covering estimate of Section~\ref{sec:l2-estimate}. In particular, we do not introduce the regularity parameters of \cite{GanHeLiWu}. Instead, we record the number of wave packets before and after each pruning step. Together with the covering argument, this suffices for the required $L^2$ estimate.

Let $\Kc=R^{\varepsilon^{100}}$. We fix compatible grids of cubes at the scales
$K^2$, $R^{1/2}$, and $R/\Kc$, changing these scales by bounded factors
if necessary. When we count incidences, a $K^2$-ball $B$ is a half-open cube of the finest grid; in the analytic estimates we use balls of comparable size. Each $K^2$-ball $B$ lies
in a unique $R^{1/2}$-cube $Q$, and each $Q$ lies in a unique horizontal
region
\[
\omega=\mathbb R^n\times I,\qquad |I|\sim R/\Kc.
\]
We write $\Om$ for the partition into horizontal regions. From the outset, all planks are enlarged by the fixed factors that are needed to neglect the tails of the wave packets.

A \emph{shaded incidence triple} $(\bB,\bT;\cG)$ consists of a collection
$\bB$ of $K^2$-balls, a family $\bT$ of $R$-planks, and a map
$\cG:\bT\to P(\Om)$, where $P(\Om)$ is the set of all subsets of $\Om$. The union $\bigcup\cG(T)$ is the shading of the plank $T$: only the part of $T$ inside it is taken into account. Set
\[
\bT(B;\cG)
=\{T\in\bT:T\cap B\neq\varnothing,
\ B\subset\textstyle\bigcup\cG(T)\},
\]
\[
\bB(T;\cG)
=\{B\in\bB:T\cap B\neq\varnothing,
\ B\subset\textstyle\bigcup\cG(T)\},
\]
and define
\[
\mathcal I(\bB,\bT;\cG)
=\sum_{B\in\bB}\#\bT(B;\cG),
\qquad
L^p(\bB,\bT;\cG;f)
=\sum_{B\in\bB}
\left\|\sum_{T\in\bT(B;\cG)}
\mathcal F^\lambda f_T\right\|_{L^p(B)}^p.
\]

We use the broad cardinality of \cite[Definition 6.3]{GanHeLiWu}:
\[
\#_{\mathrm{br},A}\bT
:=
\min_{\tau_1,\ldots,\tau_A\in\Theta_{K^{-1}}}
\#\Bigl\{T\in\bT:
\theta(T)\not\subset\bigcup_{j=1}^A3\tau_j\Bigr\}.
\]
We write $\#_{\mathrm{br}}=\#_{\mathrm{br},1}$. Thus $\#_{\mathrm{br},A}\bT$ is the least number of planks of $\bT$ that remain after the planks with directions in $A$ enlarged $K^{-1}$-caps have been removed. The cutoffs $P_\tau$ to the $K^{-1}$-caps are those fixed in Section~\ref{subsec:broadnorm}.

Fix an integer $A$ that is sufficiently large for the broad estimates below,
with $A\geq(C_n\log K)^2$, and put
\[
A_*=C_*A.
\]
Here $C_*$ is a fixed integer so large that, for any $A$ caps $\tau_1,\ldots,\tau_A\in\Theta_{K^{-1}}$, at most $A_*$ of the $K^{-1}$-caps meet $\bigcup_{j}3\tau_j$. At a given base point, these $K^{-1}$-caps can be excluded from a broad norm: for each of them, one chooses a $(k-1)$-dimensional subspace that contains one of its Gauss directions.

Suppose that $\tau_1,\ldots,\tau_A$ attain the minimum in the definition of $\#_{\mathrm{br},A}\bT$, and that the $K^{-1}$-caps meeting $\bigcup_j3\tau_j$ are excluded in this way. Then each of the remaining $K^{-1}$-caps contains at most
$\#_{\mathrm{br},A}\bT$ wave packets of $\bT$. Thus the broad cardinality bounds the number of wave packets from above after a suitable exclusion of caps, and from below after any exclusion of $A$ caps. We shall use the first property in the refined decoupling estimate and the second in the transverse covering estimate, with the same value of the broad cardinality.

\subsubsection{Algorithm 1: a single refinement step}
Let $(\bB,\bT;\cG)$ be a shaded incidence triple for which the local $L^p$ norms
$\bigl\|\sum_{T\in\bT(B;\cG)}\mathcal F^\lambda f_T\bigr\|_{L^p(B)}$, $B\in\bB$, are comparable to each other. The refinement step has three parts.

\emph{Step 1: the multiplicity $\mu$.} We sort the balls $B\in\bB$ according to the broad cardinality
$\#_{\mathrm{br},A}\bT(B;\cG)$, which is either zero or lies in a dyadic interval $[\mu,2\mu)$ with $\mu\geq1$, and we select a class that carries the largest part of $L^p(\bB,\bT;\cG;f)$.
Suppose first that this is the class with broad cardinality zero. For each ball $B$ of this class, the wave packets of $\bT(B;\cG)$ have their directions in the union of $A$ enlarged $K^{-1}$-caps. Hence the contribution of this class to the $BL^p_{k,A_*}$ norm
vanishes, up to negligible tails, and the class is treated by the induction for the narrow case in \cite[Section 8.3]{GanHeLiWu}, after localization to the horizontal regions. We call such a class a \emph{terminal narrow class}. Otherwise we obtain $\mu\geq1$ and
$\bB_1\subset\bB$ such that
\[
\mu\leq\#_{\mathrm{br},A}\bT(B;\cG)<2\mu
\qquad(B\in\bB_1),
\]
and such that $\bB_1$ carries the quantity $L^p(\bB,\bT;\cG;f)$ up to a logarithmic factor.

\emph{Step 2: pruning of the wave packets.} For dyadic numbers $w$ and $\beta$, define
\[
\begin{aligned}
\cG_w(T)
&=\bigl\{\omega\in\cG(T):
\#\{B\in\bB_1(T;\cG):B\subset\omega\}\sim w\bigr\},
\\
\bT_{w,\beta}
&=\{T\in\bT:\#\cG_w(T)\sim\beta\}.
\end{aligned}
\]
Thus $\cG_w(T)$ consists of the horizontal regions of the shading in which $T$ meets about $w$ balls of $\bB_1$, and $\bT_{w,\beta}$ consists of the planks that have about $\beta$ such regions. (In \cite[Section~6]{GanHeLiWu} the parameter $w$ is denoted by $\lambda$.) By the triangle inequality and dyadic pigeonholing, there are $w$, $\beta$,
and a subcollection $\bB_2\subset\bB_1$ such that the norms
\[
\left\|\sum_{T\in\bT_{w,\beta}(B;\cG_w)}
\mathcal F^\lambda f_T\right\|_{L^p(B)}
\]
are comparable to each other for $B\in\bB_2$, and
\[
L^p(\bB,\bT;\cG;f)
\lessapprox L^p(\bB_2,\bT_{w,\beta};\cG_w;f).
\]

\emph{Step 3: the cubes $Q$.} For each cube $Q$ of the $R^{1/2}$-grid, consider the wave packets that pass through the balls of $\bB_2$ in $Q$, before and after the pruning:
\[
\mathcal U_Q
=\bigcup_{\substack{B\in\bB_2\\B\subset Q}}
\bT(B;\cG),
\qquad
\mathcal V_Q
=\bigcup_{\substack{B\in\bB_2\\B\subset Q}}
\bT_{w,\beta}(B;\cG_w).
\]
We pigeonhole the cubes according to the dyadic values of $\#_{\mathrm{br}}\mathcal U_Q$ and $\#_{\mathrm{br}}\mathcal V_Q$. This gives a collection
$\cQ$ of cubes and dyadic numbers $m,\bar m$ such that
\[
\#_{\mathrm{br}}\mathcal U_Q\sim\bar m,
\qquad
\#_{\mathrm{br}}\mathcal V_Q\sim m
\qquad(Q\in\cQ).
\]
If $m=0$, this is another terminal narrow class: in
each cube of $\cQ$, all wave packets of $\mathcal V_Q$ have their directions in one enlarged
$K^{-1}$-cap. Otherwise $m\geq1$. We keep every ball of $\bB_2$ that lies
in a cube of $\cQ$, and set
\[
\bB_3=\{B\in\bB_2:B\subset Q\text{ for some }Q\in\cQ\}.
\]
Then
\[
L^p(\bB_2,\bT_{w,\beta};\cG_w;f)
\lessapprox
L^p(\bB_3,\bT_{w,\beta};\cG_w;f).
\]
Finally, define
\[
l=\sup_{T\in\bT_{w,\beta}}
\#\{Q\in\cQ:T\cap Q\neq\varnothing\}.
\]

The construction gives
\[
\mu\lesssim\bar m,\qquad m\lesssim\bar m,
\qquad
\mathcal I(\bB_1,\bT;\cG)\geq\#\bB_1\,\mu,
\]
and
\[
\mathcal I(\bB_1,\bT_{w,\beta};\cG_w)
\lesssim\#\bT_{w,\beta}\,\beta w.
\]
Indeed, if $B\in\bB_2$ and $B\subset Q$, then $\bT(B;\cG)\subset\mathcal U_Q$, and $\mathcal V_Q\subset\mathcal U_Q$. Since the broad cardinality is monotone under inclusion and $\#_{\mathrm{br},A}\leq\#_{\mathrm{br}}$, this gives the first two inequalities. The third holds because every $B\in\bB_1$ satisfies $\#\bT(B;\cG)\geq\mu$, and the fourth because every $T\in\bT_{w,\beta}$ has about $\beta$ regions in $\cG_w(T)$, each of which contains about $w$ balls of $\bB_1(T;\cG)$.
\begin{lemma}[{\cite[Lemma 6.5]{GanHeLiWu}}]\label{hairbrush}
Suppose that $\beta>100$. Then
\[
\#\bB_1\gtrsim(K\Kc)^{-O(1)}wlm.
\]
\end{lemma}
The proof of this lemma in \cite{GanHeLiWu} is a hairbrush argument. It uses only the pruned family $\bT_{w,\beta}$ and the number $m$; in particular, it does not require that the cubes of $\cQ$ contain comparable numbers of selected balls.

\subsubsection{Algorithm 2: iteration}
We start with the family $\mathbb T^{(0)}$ fixed above and with
$\cG^{(0)}(T)=\Om$, and we select a collection $\cB^{(0)}$ of $K^2$-balls on which
the local $L^p$ norms are comparable to each other and which carries the
total $L^p$ mass on $B_R^{n+1}$ up to a logarithmic factor. We then apply Algorithm~1 repeatedly, with
\[
\cB^{(i+1)}=\cB_3^{(i)},\qquad
\mathbb T^{(i+1)}=
\mathbb T_{w^{(i)},\beta^{(i)}}^{(i)},\qquad
\cG^{(i+1)}=\cG_{w^{(i)}}^{(i)}.
\]
If a terminal narrow class occurs, it is treated as described above. Otherwise we take the numbers $\mu^{(i)}$, $m^{(i)}$, and $\bar m^{(i)}$ to be the lower endpoints of the corresponding dyadic intervals. Since the balls of $\cB^{(i+1)}=\cB_3^{(i)}$ belong to $\cB_2^{(i)}$, the families of Step~3 satisfy $\mathcal U_Q^{(i+1)}\subset\mathcal V_Q^{(i)}\subset\mathcal U_Q^{(i)}$ for every cube $Q$, and therefore
\[
\bar m^{(i+1)}\leq m^{(i)}\leq\bar m^{(i)}.
\]
For the same reason, the incidence counts
\[
I_i:=\mathcal I(\cB_1^{(i)},\mathbb T^{(i)};\cG^{(i)})
\]
are nonincreasing in $i$. Both sequences consist of positive numbers and are bounded above by $R^{O(1)}$.
Let $\ka=R^{\varepsilon^{500}}$. It follows that there is an index
$J\lesssim_\varepsilon1$ for which
\[
\bar m^{(J)}\leq\ka\bar m^{(J+1)}
\leq\ka m^{(J)},\qquad I_J\leq\ka I_{J+1}.
\]
Indeed, if there were no such index, then at every step at least one of the two sequences would
decrease by a factor larger than $\ka$, and this can happen only
$O_\varepsilon(1)$ times.

We fix such an index $J$. We call the family $\mathbb T^{(J)}$ with its shading $\cG^{(J)}$ the \emph{input} of step $J$, and the family $\mathbb T^{(J+1)}$ with its shading $\cG^{(J+1)}$ the \emph{output}. For the input we write
\[
\mathbb T^{\mathrm{in}}:=\mathbb T^{(J)},\qquad
\cG^{\mathrm{in}}:=\cG^{(J)},\qquad
\bB_2^{\mathrm{in}}:=\cB_2^{(J)},
\]
and for the output
\[
\mathbb T:=\mathbb T^{(J+1)},\qquad
\cB:=\cB_3^{(J)},\qquad
\cG:=\cG^{(J+1)},\qquad
\cQ:=\cQ^{(J)}.
\]
We omit the superscript $J$ from $\mu,w,\beta,m,\bar m,l$.
In particular,
\begin{equation}\label{eq:stable-input-count}
\#_{\mathrm{br}}
\bigcup_{\substack{B\in\bB_2^{\mathrm{in}}\\B\subset Q}}
\mathbb T^{\mathrm{in}}(B;\cG^{\mathrm{in}})
\sim\bar m\lesssim\ka m
\qquad(Q\in\cQ).
\end{equation}
Moreover, every $B\in\cB$ satisfies
\[
\mu\leq
\#_{\mathrm{br},A}
\mathbb T^{\mathrm{in}}(B;\cG^{\mathrm{in}})
<2\mu.
\]
Note that these two statements concern the input family $\mathbb T^{\mathrm{in}}$, which in general is larger than the output family $\mathbb T$.

By the almost orthogonality of wave packets and the normalization of $\mathbb T^{(0)}$,
\begin{equation}\label{eq:packet-energy}
(\#\mathbb T)R^{\frac{n-1}{2}}
\sim\sum_{T\in\mathbb T}\|f_T\|_2^2
\lesssim\|f\|_2^2.
\end{equation}
Suppose that $\beta>100$. By the choice of $J$ and the inequalities recorded after Step~3,
\[
\#\cB_1^{(J)}\mu
\leq I_J\leq\ka I_{J+1}
\leq\ka\mathcal I(\cB_1^{(J)},\mathbb T;\cG)
\lesssim\ka(\#\mathbb T)\beta w.
\]
We combine this with Lemma~\ref{hairbrush}. Since $\beta\lesssim\Kc$, and since $\Kc$ and $\ka$ are bounded by $K$, we obtain
\begin{equation}\label{eq:hairbrush-cancellation}
lm\mu\lesssim K^{O(1)}\#\mathbb T.
\end{equation}

\subsubsection{Completing the reduction}\label{subsubsec:completingreduction}
The case
$\beta\leq100$ is the one-end case, which is treated by induction in
\cite[Section 8.1]{GanHeLiWu}. We therefore assume from now on that $\beta>100$ and
that no terminal narrow class has occurred, so that $\mu\geq1$ and $m\geq1$. We subdivide each horizontal region into cubes $B_0$ of side length
comparable to $R/\Kc$, compatible with the $K^2$-grid.

Choose a cube $B_0$ that maximizes
\[
\sum_{\substack{B\in\cB\\B\subset B_0}}
\left\|\sum_{T\in\mathbb T(B;\cG)}
\mathcal F^\lambda f_T\right\|_{L^p(B)}^p,
\]
and let $\omega_0$ be the horizontal region that contains it. Define
\[
X:=\bigcup_{\substack{B\in\cB\\B\subset B_0}}B,
\qquad
\mathbb T_0:=
\{T\in\mathbb T:T\cap B_0\neq\varnothing,
\ \omega_0\in\cG(T)\},
\qquad
f_0:=\sum_{T\in\mathbb T_0}f_T.
\]
If $B\subset X$ is a ball of $\cB$, then every plank of $\mathbb T_0$ that meets $B$ belongs
to $\mathbb T(B;\cG)$, and hence to
$\mathbb T^{\mathrm{in}}(B;\cG^{\mathrm{in}})$.

Let $B\subset X$ be a ball of $\cB$, and let $\tau_1,\ldots,\tau_A$ attain the minimum in the definition of
$\#_{\mathrm{br},A}\mathbb T^{\mathrm{in}}(B;\cG^{\mathrm{in}})$. We exclude the $K^{-1}$-caps that meet $\bigcup_j3\tau_j$
by means of at most $A_*$ subspaces, as described above. Each of the remaining
$K^{-1}$-caps then contains fewer than $2\mu$ wave packets of $\mathbb T_0$
that meet $B$. Thus the number $\mu$ serves two purposes: it is the upper
bound for the multiplicity in the refined decoupling estimate, and it is the lower bound for the multiplicity in the transverse covering argument.

By the two-ends and broad--narrow reductions of
\cite[Section~8]{GanHeLiWu}, the proof of Theorem~\ref{thm:refined-GHLW-density} is now reduced to the estimate
\begin{equation}\label{eq:reduced-density-target}
\|\mathcal F^\lambda f_0\|_{BL_{k,3A_*}^p(X)}
\lesssim
R^{(n-1)(\frac12-\frac1p)+O(\varepsilon^2)}
\|f\|_2^{2/p}W(f,B_R^{n+1})^{1-2/p}.
\end{equation}
The broadness parameter $3A_*$ is a fixed multiple of $A$, which is admissible in that reduction; see the end of Section~\ref{subsec:broadnorm}. Note that the $L^2$ norm and the density on the
right-hand side are those of the original function $f$, and not those of $f_0$. This is convenient, because the lower bounds of Lemma~\ref{lem:density-lower-bounds} for $W(f,B_R^{n+1})$ may be applied to any subfamily of the wave packets of $f$, and in particular to the family $\mathbb T$.

\subsection{The transverse covering \texorpdfstring{$L^2$}{L2} estimate}\label{sec:l2-estimate}
We now prove the $L^2$ estimate. For each $K^{-1}$-cap $\tau$, we use the lower bound $\mu$ for the multiplicity to cover the set $X$ by planks whose directions are transverse to $\tau$. The covering planks are taken from the input family $\mathbb T^{\mathrm{in}}$. The function that we estimate is built from the output family $\mathbb T$ only, so its energy is bounded by $\|f\|_2^2$.

\begin{lemma}\label{lem:transverse-covering}
With the notation of Section~\ref{sec:two-end-reduction},
\[
\|\mathcal F^\lambda f_0\|_{BL_{k,A_*}^{2}(X)}^2
\lessapprox K^{O(1)}\frac{lm}{\mu}\|f\|_2^2.
\]
\end{lemma}

\begin{proof}
\emph{The covering planks.} For $Q\in\mathcal Q$, let
\[
\mathcal U_Q
:=\bigcup_{\substack{B\in\mathbf B_2^{\mathrm{in}}\\B\subset Q}}
\mathbb T^{\mathrm{in}}(B;\mathcal G^{\mathrm{in}}).
\]
Choose a $K^{-1}$-cap $\tau_Q$ that attains the minimum in the definition of $\#_{\mathrm{br}}\mathcal U_Q$,
and set
\[
\mathcal H_Q:=\{T'\in\mathcal U_Q:
\theta(T')\not\subset3\tau_Q\}.
\]
By the choice of $\tau_Q$ and \eqref{eq:stable-input-count},
\[
\#\mathcal H_Q\sim\bar m\lesssim\kappa m.
\]

\emph{The function to be estimated.} Let $\mathcal S_Q$ consist of those $T\in\mathbb T_0$ that meet at least
one ball $B\subset X\cap Q$ of $\cB$. For each $K^{-1}$-cap $\tau$, define
\[
g_{Q,\tau}:=\sum_{T\in\mathcal S_Q}P_\tau f_T,
\qquad
E_{Q,\tau}^2:=\sum_{T\in\mathcal S_Q}\|P_\tau f_T\|_2^2.
\]
Since wave packets are localized to their planks, we may replace $f_{0,\tau}=P_\tau f_0$ by
$g_{Q,\tau}$ on $X\cap Q$, up to rapidly decaying errors. Moreover, since the cutoffs $P_\tau$ have bounded overlap and every plank of $\mathbb T$ meets at most $l$ cubes of $\mathcal Q$,
\begin{equation}\label{eq:signal-energy-count}
\sum_{Q\in\mathcal Q}\sum_\tau E_{Q,\tau}^2
\lesssim
\sum_{T\in\mathbb T_0}
\#\{Q\in\mathcal Q:T\cap Q\neq\emptyset\}\|f_T\|_2^2
\leq l\sum_{T\in\mathbb T}\|f_T\|_2^2
\lesssim l\|f\|_2^2.
\end{equation}

\emph{The covering.} Fix $Q$ and $\tau$, and let $C_0$ be a sufficiently large constant that depends only on the bounds in the normalization of the phase. The union of
$3\tau_Q$ and $C_0\tau$ is covered by the enlargements of
$O_{n,C_0}(1)$ of the $K^{-1}$-caps. Since $A$ is sufficiently large, these
caps may all be included among the caps $\tau_1,\ldots,\tau_A$ in the definition of
$\#_{\mathrm{br},A}$. Now let $B\subset X\cap Q$ be a ball of $\cB$. Since $\#_{\mathrm{br},A}\mathbb T^{\mathrm{in}}(B;\cG^{\mathrm{in}})\geq\mu$, the ball $B$ meets
at least $\mu$ wave packets of the input family whose
directions lie outside this union. These wave packets belong to $\mathcal H_Q$, and their directions are separated from $\tau$.

Let $\mathcal H_{Q,\tau}$ denote the set of wave packets in $\mathcal H_Q$ with
this angular separation, and put $T'^+=N_{CK^2}(T')$, where $C$ is
so large that $T'\cap B\neq\emptyset$ implies $B\subset T'^+$.
Then every point of $X\cap Q$ lies in at least $\mu$ of the sets $T'^+$, $T'\in\mathcal H_{Q,\tau}$; that is,
\begin{equation}\label{eq:transverse-covering}
\mu\,1_{X\cap Q}
\lesssim
\sum_{T'\in\mathcal H_{Q,\tau}}1_{T'^+\cap Q}.
\end{equation}

\emph{The estimate on one plank.} We claim that, for every $T'\in\mathcal H_{Q,\tau}$,
\begin{equation}\label{eq:local-transverse-trace}
\int_{T'^+\cap Q}|\mathcal F^\lambda g_{Q,\tau}|^2
\lesssim K^{O(1)}E_{Q,\tau}^2.
\end{equation}
To prove this, let $z_Q$ be the center of $Q$, let $\xi_{T'}$ be the center of the cap $\theta(T')$, and set
\[
v_{T'}:=\partial_z\phi^\lambda(z_Q;\xi_{T'}).
\]
By homogeneity and the definition of a wave packet plank, the values of
$\phi^\lambda(z;\xi_{T'})$ for $z\in T'$ lie in an interval of length
$R^{O(\delta)}$. We expand this function to first order about $z_Q$; the remainder is
$O(R/\lambda)$ on $Q$. It follows that $T'^+\cap Q$ is contained in an affine slab
with normal $v_{T'}$ and thickness $K^{O(1)}$, whose $n$ remaining side
lengths are $O(R^{1/2})$.

The Gauss directions of $\tau$ are transverse to this slab. To see this, we use the local representation of the cone from the proof of
\cite[Lemma 5.7]{GanHeLiWu}: at a
fixed point, the normalized cone of positive curvature can be written
in angular coordinates as
\[
v(\eta)=(\eta,1,h(\eta)),
\qquad \nabla^2h\asymp I,
\]
while its Gauss direction is parallel to
\[
G_0(\zeta)
=
\bigl(\nabla h(\zeta),
 h(\zeta)-\zeta\cdot\nabla h(\zeta),-1\bigr).
\]
By convexity,
\[
\begin{aligned}
|v(\eta)\cdot G_0(\zeta)|
&=h(\eta)-h(\zeta)-\nabla h(\zeta)\cdot(\eta-\zeta)\gtrsim |\eta-\zeta|^2.
\end{aligned}
\]
The direction of $T'$ lies outside $C_0\tau$, so it is separated from every direction in $\tau$ by at least a constant multiple of $C_0K^{-1}$. Hence, if $C_0$ is large enough,
\[
\angle\bigl(G^\lambda(z_Q;\tau),v_{T'}^\perp\bigr)>2K^{-2}.
\]
The variation of the Gauss map across $Q$ is
$O(R^{1/2}/\lambda)=o(K^{-2})$, so the angle remains larger than $K^{-2}$ when $z_Q$ is replaced by any point of the slab. We cover the slab by $K^{O(1)}$ slabs of thickness $1$ and apply Lemma~\ref{lem:GHLW-transverse-l2} on each of them
to the wave packets $P_\tau f_T$, $T\in\mathcal S_Q$. This proves
\eqref{eq:local-transverse-trace}.

\emph{Conclusion.} We integrate $|\mathcal F^\lambda g_{Q,\tau}|^2$ against \eqref{eq:transverse-covering} and apply
\eqref{eq:local-transverse-trace} to each term. Since
$\#\mathcal H_{Q,\tau}\lesssim\kappa m$ and $\kappa\leq K$, we obtain
\[
\int_{X\cap Q}|\mathcal F^\lambda g_{Q,\tau}|^2
\lesssim K^{O(1)}\frac{m}{\mu} E_{Q,\tau}^2.
\]
The broad $L^2$ norm is bounded by the square sum over the $K^{-1}$-caps.
Summing over $Q$ and $\tau$, and then applying
\eqref{eq:signal-energy-count}, we conclude that
\[
\|\mathcal F^\lambda f_0\|_{BL_{k,A_*}^{2}(X)}^2
\lesssim
\sum_{Q\in\mathcal Q}\sum_\tau
\int_{X\cap Q}|\mathcal F^\lambda g_{Q,\tau}|^2\lessapprox K^{O(1)}\frac{lm}{\mu}\|f\|_2^2.
\qedhere
\]
\end{proof}

Lemma~\ref{lem:transverse-covering} provides the $L^2$ factor in the H\"older inequality
\eqref{eq:broad-holder} for broad norms; it is
combined with the other estimates in
Section~\ref{sec:interpolation}.

The proof of Lemma~\ref{lem:transverse-covering} does not depend
on the parity of $n$ or on the value of $k$, and it applies in every dimension $n\geq3$.
For $n=3$ it replaces the $L^2$ argument of
\cite[Section~9.2]{GanHeLiWu}, in which a regularity condition on slabs is needed, by a direct covering
argument. For odd $n\geq3$, take $k=2$. Combining Lemma~\ref{lem:transverse-covering}
with the bound \eqref{eq:hairbrush-cancellation} and the standard global
$L^2$ estimate, and then interpolating with the refined
decoupling estimate, one recovers the odd-dimensional result of \cite{GanHeLiWu}.

\subsection{Interpolation}\label{sec:interpolation}
Set
\[
q=\frac{2(n+1)}{n-1},\qquad
d=n+k+1,\qquad r=\frac{2d}{d-2}.
\]
We prove \eqref{eq:reduced-density-target} by interpolating between three estimates for $\mathcal F^\lambda f_0$ on $X$, at the exponents $q$, $r$, and $2$. The first is the refined decoupling estimate of Gan, He, Li, and Wu, in the form that it has before the number of wave packets is expressed in terms of the wave packet density.

\begin{lemma}[{\cite[Section~9.1, (9.4)]{GanHeLiWu}}]
\label{lem:refined-decoupling}
In the situation of Section~\ref{subsubsec:completingreduction},
\[
\|\mathcal F^\lambda f_0\|_{BL_{k,A_*}^{q}(X)}
\lessapprox K^{O(1)}
\mu^{\frac1{n+1}}(\#\mathbb T)^{1/q}R^{\frac{n-1}{4}}.
\]
\end{lemma}
This form of the estimate follows from the upper bound for the multiplicity established in
Section~\ref{subsubsec:completingreduction} and from the normalization of the $L^2$ norms of the wave packets.

The second estimate is the local form, on a ball of radius $R$, of Schippa's global broad estimate.
\begin{lemma}[{\cite[Theorem~1.2]{Schippa2024}}]
\label{lem:global-broad-endpoint}
If $A_*$ is sufficiently large, then for every $g$ whose Fourier transform is supported
in the sector \eqref{supp-f-hat},
\[
\|\mathcal F^\lambda g\|_{BL_{k,A_*}^{r}(B_R)}
\lessapprox\|g\|_2.
\]
\end{lemma}

We choose the parameter $A$ so large that $A_*$
is an admissible broadness parameter in Lemma~\ref{lem:global-broad-endpoint}. Restricting the estimate of the lemma to $X$ and using the almost orthogonality of wave packets, we get
\[
\|\mathcal F^\lambda f_0\|_{BL_{k,A_*}^{r}(X)}
\lessapprox\|f_0\|_2\lesssim\|f\|_2.
\]
The third estimate is Lemma~\ref{lem:transverse-covering}, which together with
\eqref{eq:hairbrush-cancellation} gives
\[
\|\mathcal F^\lambda f_0\|_{BL_{k,A_*}^{2}(X)}
\lessapprox K^{O(1)}
\|f\|_2\frac{(\#\mathbb T)^{1/2}}{\mu}.
\]

Let $\theta_q,\theta_r,\theta_2$ be nonnegative weights satisfying
\[
\theta_q+\theta_r+\theta_2=1,\qquad
\frac1p=\frac{\theta_q}{q}+\frac{\theta_r}{r}+\frac{\theta_2}{2}.
\]
Applying the H\"older inequality \eqref{eq:broad-holder} for broad norms to
the three estimates above, we obtain
\[
\|\mathcal F^\lambda f_0\|_{BL_{k,3A_*}^{p}(X)}
\lessapprox{}
K^{O(1)}R^{\frac{n-1}{4}\theta_q}
\|f\|_2^{\theta_r+\theta_2}
\mu^{\frac{\theta_q}{n+1}-\theta_2}
(\#\mathbb T)^{\frac{\theta_q}{q}+\frac{\theta_2}{2}}.
\]
We choose
\[
\theta_q=(n+1)\theta_2,\qquad
\theta_r=\frac{d(n-3)}4\theta_2,\qquad
\theta_2=\frac4{4(n+2)+d(n-3)}.
\]
With $\delta_p=\frac12-\frac1p$, these weights satisfy
\[
\delta_p
=\frac{\theta_q}{n+1}+\frac{\theta_r}{d}
=\frac{n+1}{4}\theta_2.
\]
The powers of $\mu$ cancel, and we are left with
\begin{equation}\label{eq:interpolated-packet-bound}
\|\mathcal F^\lambda f_0\|_{BL_{k,3A_*}^{p}(X)}
\lessapprox
K^{O(1)}R^{(n-1)\delta_p}
\|f\|_2^{1-(n+1)\theta_2}
(\#\mathbb T)^{n\theta_2/2}.
\end{equation}

It remains to express the number of wave packets in terms of the density. Lemma~\ref{lem:density-lower-bounds}, applied to the family $\mathbb T$, gives
\[
W(f,B_R^{n+1})^2\gtrsim
\max\left\{1,\frac{\#\mathbb T}{R^{(n+1)/2}}\right\}.
\]
Together with \eqref{eq:packet-energy}, this implies
\[
\|f\|_2 W(f,B_R^{n+1})
\gtrsim
\max\left\{
(\#\mathbb T)^{1/2}R^{\frac{n-1}{4}},
(\#\mathbb T)R^{-1/2}
\right\}\geq(\#\mathbb T)^{n/(n+1)}.
\]
In the last step we bounded the maximum from below by the weighted geometric mean of the two terms, with weight
$2/(n+1)$ on the first. Since
$2\delta_p=(n+1)\theta_2/2$, we obtain
\[
\|f\|_2^{1-(n+1)\theta_2}
(\#\mathbb T)^{n\theta_2/2}\lesssim
\|f\|_2^{2/p}W(f,B_R^{n+1})^{1-2/p}.
\]
Combined with \eqref{eq:interpolated-packet-bound}, this proves
\eqref{eq:reduced-density-target}, provided that the weights correspond to the exponent $p=p_*(n)$ and that the losses are acceptable. We now check this.

For $k=(n+4)/2$, we have $d=3(n+2)/2$ and
\[
\theta_q=\frac{8(n+1)}{(n+2)(3n-1)},\qquad
\theta_r=\frac{3(n-3)}{3n-1},\qquad
\theta_2=\frac8{(n+2)(3n-1)}.
\]
These weights are positive for every $n\geq4$, and the
corresponding exponent is
\[
p=p_*(n)
=2+\frac{8(n+1)}{3n^2+n-6}
=2+\frac8{3n-2-\frac4{n+1}}.
\]
Moreover,
\[
p_*(n)>2+\frac8{3n-2}
=2+\frac4{2n-k+1},
\]
so the condition of Section~\ref{subsec:broadnorm} for the passage from broad to linear estimates is satisfied. The factors
$K^{O(1)}$ and $\kappa^{O(1)}$ and the arbitrarily small powers of $R$
are bounded by $R^{O(\varepsilon^2)}$. Thus \eqref{eq:reduced-density-target} holds, and \eqref{wave-packet-density-estimate} follows for $p=p_*(n)$.

The estimate extends to $p>p_*(n)$ by interpolation with
\[
\|\mathcal F^\lambda f\|_{L^\infty(B_R^{n+1})}
\lessapprox R^{(n-1)/2}W(f,B_R^{n+1}).
\]
This bound is obtained by testing the density at the finest scale $s=R^{-1/2}$, applying the Cauchy--Schwarz inequality to each wave packet, and using that the
wave packets with a fixed direction have bounded overlap. The proof of
Theorem~\ref{thm:refined-GHLW-density} is complete.

\section{Limitations: proofs of Theorems~\ref{thm:counterexample} and~\ref{thm:counter-example-multiscale-bound}}\label{sec:counterexample}
In this section we prove Theorems~\ref{thm:counterexample} and
\ref{thm:counter-example-multiscale-bound}. Both proofs use the same family of wave packets for the Euclidean wave equation. The wave packets are tangent to a hyperboloid: they overlap heavily near it, but their wave packet density is as small as that of a single wave packet. This family violates the wave packet density estimate below $p_n^{+}$, and in three spatial dimensions it still violates the estimate for the multiscale density in a smaller range of exponents.

The construction arose from an attempt to improve Euclidean local smoothing bounds by combining wave packet density with the polynomial Wolff axioms. In that attempt, the configurations that resisted improvement were large families of wave packets tangent to a hypersurface of hyperbolic type. We then observed that such families already violate the wave packet density estimate below $p_n^{+}$.

Configurations of this kind are already implicit in the author's earlier paper \cite{WangCone}, although they are not singled out there: the last paragraph of \cite[Section~1.1]{WangCone} and the adjustment of the nested polynomial Wolff axiom in \cite[Section~2.9]{WangCone} were made in order to deal with similar situations.

\subsection{Why the density estimate can fail}\label{subsec:crossing}
Each test in the definition of $W$ concerns one cap $\tau\in\Theta_s$ and one box $V\parallel V_{\tau,R}$, and it counts only the wave packets $T$ with $\theta(T)\subset\tau$ and $T^\flat\subset V$. Wave packets in other directions are not counted in this test, even when they pass through the same region; see Figure~\ref{fig:crossing}. Such wave packets are counted in tests with wider caps, but the box $V_{\tau,R}$ grows with the cap, and the additional wave packets are weighed against a larger denominator $|V|$. Thus the density measures the concentration of wave packets with nearly the same direction, and it is largely blind to concentration that is produced by many crossing directions.

\begin{figure}[ht]
\centering
\begin{tikzpicture}[x=1cm,y=1cm,font=\small]
  \fill[black!4] (-2.7,-0.9) rectangle (2.7,0.9);
  \draw[densely dashed,black!70] (-2.7,-0.9) rectangle (2.7,0.9);
  \node[anchor=south west,inner sep=2pt] at (-2.7,0.9) {$V\parallel V_{\tau,R}$};
  \foreach \y in {-0.55,0,0.55}
    \filldraw[fill=teal!35,draw=teal!60!black,line width=0.3pt]
      (-2.4,\y-0.08) rectangle (2.4,\y+0.08);
  \foreach \a in {24,50,76,104,130,156}
    \filldraw[rotate=\a,fill=orange!80!black,draw=orange!40!black,
      fill opacity=0.85,line width=0.3pt] (-1.8,-0.065) rectangle (1.8,0.065);
  \draw[black!55,line width=0.3pt] (2.25,0.55) -- (3.2,0.95)
    node[right,inner sep=1pt,text=black] {$\theta(T)\subset\tau,\ T^\flat\subset V$};
  \draw[black!55,line width=0.3pt] (0.44,-1.69) -- (1.3,-1.75)
    node[right,inner sep=1pt,text=black] {$\theta(T)\not\subset\tau$};
\end{tikzpicture}
\caption{A schematic spatial slice of base boxes for one test $(\tau,V)$ in the definition of $W$. The test counts the wave packets with $\theta(T)\subset\tau$ and $T^\flat\subset V$ (light). Wave packets in other directions (dark) are not counted in this test, even though they pass through the same region. The picture is not to scale.}
\label{fig:crossing}
\end{figure}

In extreme cases the density is much smaller than the supremum of the solution. Consider a bush: about $R^{(n-1)/2}$ wave packets $f_T$ with $\|f_T\|_2\sim1$, one for each $R^{-1/2}$-cap, whose planks pass through a common point. A test with an $s$-cap $\tau$ counts at most about $(R^{1/2}s)^{n-1}$ of them, while $|V|\sim R^ns^{n+1}$. Hence $W(f,B_R^{n+1})^2\lesssim R^{-(n-1)/2}$, which is the density of a single wave packet. On the other hand, if the wave packets interfere constructively at the common point, then the solution has size about $R^{(n-1)/2}\cdot R^{-(n-1)/4}=R^{(n-1)/4}$ near this point, that is, $R^{(n-1)/2}$ times the amplitude of a single wave packet. Thus the bound $\|\mathcal F^\lambda f\|_{L^\infty(B_R^{n+1})}\lessapprox R^{(n-1)/2}W(f,B_R^{n+1})$ from the end of Section~\ref{sec:interpolation} is attained. A bush does not contradict the wave packet density estimate, because its wave packets overlap heavily only near one point. The family constructed below produces such crossings along a hyperboloid.

\subsection{Proof of Theorem~\ref{thm:counterexample}}\label{subsec:proof-counterexample}
We treat odd and even dimensions together. Define
\[
(l,\beta)
=
\begin{cases}
\left(\dfrac{n+3}{2},\dfrac12\right),
& n \text{ odd},\\[3mm]
\left(\dfrac n2+2,0\right),
& n \text{ even}.
\end{cases}
\]
Thus $3\leq l\leq n$. (In this section the letters $l$, $\beta$, $\omega$, and $\Omega$ do not have the meaning that they had in Section~\ref{sec:refinement}.)

Fix a sufficiently small constant $c>0$, and consider the
$l$-dimensional hyperboloid
\[
\mathcal H_R
=
\left\{
(x',x'',t)
\in
\mathbb R^l\times\mathbb R^{n-l}\times\mathbb R:
|x'|^2-t^2=c^2R^2,\quad x''=0
\right\}.
\]

\noindent
\textbf{Step 1. Construction of the wave packets.}

Let $\Omega_R\subset S^{l-1}$ be a maximal $CR^{-1/2}$-separated set, where $C$ is a sufficiently
large fixed constant. Then
\begin{equation}
\label{eq:omega-number}
\#\Omega_R\sim R^{(l-1)/2}.
\end{equation}

For each $\omega\in\Omega_R$, let
\[
U_\omega
\subset
S^{l-1}\cap\omega^\perp
\cong S^{l-2}
\]
be a maximal $CR^{-1/2}$-separated set. Hence
\begin{equation}
\label{eq:u-number}
\#U_\omega
\sim R^{(l-2)/2}.
\end{equation}

Finally, let $Q_R$ be a $C$-separated subset of $[-c_1R^\beta,c_1R^\beta]$
of maximal cardinality, where $c_1>0$ is sufficiently small. Thus
\begin{equation}
\label{eq:q-number}
\#Q_R\sim R^\beta.
\end{equation}
When $n$ is even, we have $\beta=0$ and may simply take $Q_R=\{0\}.$

For
\[
\omega\in\Omega_R,
\qquad
u\in U_\omega,
\qquad
q\in Q_R,
\]
define
\[
b_{\omega,u,q}
:=
cR\,u+q\omega
\in\mathbb R^l
\subset\mathbb R^n.
\]
Since $u\perp\omega$,
\begin{equation}
\label{eq:b-properties}
b_{\omega,u,q}\cdot\omega=q,
\qquad
|b_{\omega,u,q}|^2=c^2R^2+q^2.
\end{equation}

Let $\theta_\omega$ be an $R^{-1/2}$-cap centered at the direction
$\omega$. For every triple $(\omega,u,q)$, let $f_{\omega,u,q}$ be a standard
wave packet at scale $R$ associated with the phase-space tile $(\theta_\omega,b_{\omega,u,q})$, that is, with frequency support in $\theta_\omega$ and with center $b_{\omega,u,q}$ at time $t=0$. We normalize it so that $\|f_{\omega,u,q}\|_2\sim1$. At time $t=0$, this wave packet is concentrated on a box of dimensions $1\times R^{1/2}\times\cdots\times R^{1/2}$ centered at $b_{\omega,u,q}$, whose short side is parallel to $\omega$. In spacetime, $e^{it\sqrt{-\Delta}}f_{\omega,u,q}$
is essentially supported on a plank of dimensions $1\times R^{1/2}\times\cdots\times R^{1/2}\times R$
whose central line is
\[
L_{\omega,u,q}
=
\left\{
(b_{\omega,u,q}-t\omega,0,t):
|t|\lesssim R
\right\}.
\]
We denote by $\mathbb T_R$ the family of these planks.

We claim that the planks of $\mathbb T_R$ are tangent to $\mathcal H_R$ in the usual sense: they lie in the $O(R^{1/2})$-neighborhood of $\mathcal H_R$, and their directions make an angle $O(R^{-1/2})$ with the tangent spaces of $\mathcal H_R$ at nearby points. Indeed, along the central line,
\begin{align}
|b_{\omega,u,q}-t\omega|^2-t^2-c^2R^2
&=
|b_{\omega,u,q}|^2-c^2R^2
-2t\,b_{\omega,u,q}\cdot\omega
=
q^2-2tq.
\label{eq:F-line}
\end{align}
Since
\[
|q|\lesssim R^\beta
\leq R^{1/2},
\qquad
|t|\lesssim R,
\]
we have
\[
|q^2-2tq|
\lesssim R^{3/2}.
\]
On the portion of the hyperboloid that is relevant here,
\[
|\nabla(|x'|^2-t^2-c^2R^2)|
\sim R.
\]
Therefore every point of the central line lies within distance $O(R^{1/2})$ of the hyperboloid:
\[
\operatorname{dist}
\bigl(L_{\omega,u,q},\mathcal H_R\bigr)
\lesssim R^{1/2}.
\]

Moreover, the direction of the central line is $(-\omega,0,1)$, and a normal vector to $\mathcal H_R$ at a nearby point $z=(x',0,t)\in\mathcal H_R$ is $N_z=(x',0,-t)$. Using \eqref{eq:b-properties} and \eqref{eq:F-line}, one obtains
\[
\angle\bigl((-\omega,0,1),T_z\mathcal H_R\bigr)
\lesssim
\frac{|q|}{R}+R^{-1/2}
\lesssim R^{-1/2}.
\]
Thus the planks, with their usual thickness $R^{1/2+O(\delta)}$, satisfy the
standard tangency condition. When $n$ is even, we have $q=0$, and the central lines are
rulings of the hyperboloid.

Define
\[
f_R
=
\sum_{\omega\in\Omega_R}
\sum_{u\in U_\omega}
\sum_{q\in Q_R}
\varepsilon_{\omega,u,q}
f_{\omega,u,q},
\]
where the signs $\varepsilon_{\omega,u,q}\in\{-1,1\}$ will be chosen later.
By the almost orthogonality of the wave packets and by \eqref{eq:omega-number}, \eqref{eq:u-number}, and
\eqref{eq:q-number},
\begin{equation}
\label{eq:F-unified}
\|f_R\|_2^2 \sim \#\mathbb T_R \sim
R^{l-\frac32+\beta}
=
R^{\frac{n+1}{2}}
\end{equation}
for every choice of the signs.

\noindent
\textbf{The geometry of the configuration.}

Before we estimate the density, we describe how this family realizes the situation of Section~\ref{subsec:crossing} on a large set. Figure~\ref{fig:shaft} shows the case $n=3$, in which $l=3$. At time $t=0$, the hyperboloid $\mathcal H_R$ meets the hyperplane $\{t=0\}$ in the sphere $\{|x|=cR\}\subset\mathbb R^3$. Fix a point $A$ of this sphere, and let $O$ be the origin. If the base box of a wave packet $f_{\omega,u,q}$ contains $A$, then \eqref{eq:b-properties} shows that $|A\cdot\omega|\lesssim R^{1/2}$. Hence $\omega$ is tangent to the sphere at $A$ up to an angle $O(R^{-1/2})$, and the base box, a slab of dimensions $R^{1/2}\times R^{1/2}\times1$ whose short side is parallel to $\omega$, contains, after a slight enlargement, a segment of length comparable to $R^{1/2}$ of the line $OA$. About $R^{1/2}$ directions of $\Omega_R$ are tangent to the sphere at $A$ in this sense, and the corresponding slabs rotate around the shaft $OA$. Thus about $R^{1/2}$ wave packets with $R^{-1/2}$-separated directions pass near $A$ and cross there, as in a bush. A test $(\tau,V)$ with a small cap $\tau$ counts only the few of them whose directions lie in $\tau$, while tests with wide caps come with large boxes $V$. The same picture appears at every point of the sphere. Since the planks stay in the $O(R^{1/2})$-neighborhood of $\mathcal H_R$ for a time comparable to $R$, the crossings persist in the set $U_R$ of Step~3, whose measure is comparable to $R^{(n+l+1)/2}$. For general $n$ the picture is similar: the directions that are tangent at $A$ to the sphere $\{|x'|=cR,\ x''=0\}$ form a sphere of dimension $l-2$. In Step~2 we show that the density of $f_R$ is nevertheless as small as that of a single wave packet, and in Step~3 we show that the many crossings make the $L^p$ norm of the solution large.

\begin{figure}[ht]
\centering
\begin{tikzpicture}[x=1cm,y=1cm,font=\small]
  \def\r{1.5}
  \shade[ball color=black!12] (0,0) circle (\r);
  \draw[black!35,line width=0.3pt] (0,0) circle (\r);
  \draw[black!30,dashed,line width=0.3pt] (\r,0) arc (0:180:\r cm and 0.33cm);
  \draw[black!45,line width=0.3pt] (-\r,0) arc (180:360:\r cm and 0.33cm);
  \node[anchor=south east,inner sep=1pt] at (-0.85,1.25) {$|x|=cR$};
  \coordinate (A) at (\r,0);
  \foreach \p in {0,30,60,90,120,150}{
    \pgfmathsetmacro{\ex}{-0.42*sin(\p)}
    \pgfmathsetmacro{\ey}{cos(\p)-0.28*sin(\p)}
    \filldraw[fill=teal!45,fill opacity=0.3,draw=teal!50!black,line width=0.35pt]
      ($(A)+(0.85,0)+0.85*(\ex,\ey)$) -- ($(A)+(0.85,0)-0.85*(\ex,\ey)$)
      -- ($(A)-(0.85,0)-0.85*(\ex,\ey)$) -- ($(A)-(0.85,0)+0.85*(\ex,\ey)$) -- cycle;
  }
  \draw[line width=0.9pt] (0,0) -- (3.15,0);
  \fill (0,0) circle (1.4pt);
  \fill (A) circle (1.4pt);
  \node[below left,inner sep=2pt] at (0,0) {$O$};
  \node[above,inner sep=1pt,fill=white] at (0.32,0.08) {$cR$};
  \node[above left,inner sep=1pt,fill=white] at ($(A)+(-0.04,0.06)$) {$A$};
  \node[anchor=north] at (0.8,-1.75) {(a)};
  \begin{scope}[xshift=5.7cm]
    \draw[black!35,line width=0.3pt] (0,0) circle (1.3);
    \foreach \k in {0,...,11}{
      \pgfmathsetmacro{\ang}{15*\k}
      \ifnum\k<3\relax
        \def\fc{teal!40}\def\dc{teal!55!black}
      \else
        \def\fc{orange!80!black}\def\dc{orange!40!black}
      \fi
      \filldraw[rotate=\ang,fill=\fc,draw=\dc,line width=0.3pt]
        (-1.15,-0.045) rectangle (1.15,0.045);
    }
    \fill (0,0) circle (1.6pt);
    \draw[black!55,line width=0.3pt] (0.05,-0.05) -- (0.9,-1.55)
      node[below,inner sep=1pt,text=black] {axis $OA$};
    \draw[black!70,line width=0.35pt,<->,>=stealth] (90:1.45) arc (90:105:1.45);
    \node[above,inner sep=2pt] at (97.5:1.5) {spacing $\sim R^{-1/2}$};
    \node[anchor=north] at (-0.45,-1.75) {(b)};
  \end{scope}
\end{tikzpicture}
\caption{The configuration of Section~\ref{subsec:proof-counterexample} for $n=3$ at time $t=0$; schematic, not to scale. (a) The base boxes of the wave packets that pass through a point $A$ of the sphere $\{|x|=cR\}$ are slabs of dimensions $R^{1/2}\times R^{1/2}\times1$ that contain a piece of the line $OA$. Their short sides point in directions tangent to the sphere at $A$, so the slabs rotate around the shaft $OA$. (b) Viewed along $OA$, the slabs have about $R^{1/2}$ directions, with angular spacing $\sim R^{-1/2}$. A test $(\tau,V)$ with a small cap $\tau$ counts only the few slabs whose directions lie in $\tau$ (light).}
\label{fig:shaft}
\end{figure}

\noindent
\textbf{Step 2. The wave packet density.}

Let $R^{-1/2}\leq s\leq1$, let $\tau$ be an $s$-cap, and let
\[
V\parallel V_{\tau,R},
\qquad
V_{\tau,R}
\sim
Rs^2\times(Rs)^{n-1}.
\]
Thus
\begin{equation}
\label{eq:V-volume}
|V|
\sim R^ns^{n+1}.
\end{equation}

We have to count the planks $T\in\mathbb T_R$ with $\theta(T)\subset\tau$ and $T^\flat\subset V$. We say that a center $b_{\omega,u,q}$ is \emph{compatible with $V$} if the base box of the corresponding plank is contained in $V$, and that a direction $\omega\in\Omega_R$ is compatible with $V$ if $b_{\omega,u,q}$ is compatible with $V$ for some $u\in U_\omega$ and $q\in Q_R$.

Let $\omega_\tau$ denote the center direction of $\tau$, and let
$b_V$ denote the center of $V$. Suppose that
$\theta_\omega\subset\tau$ and that $b_{\omega,u,q}$ is compatible with $V$ for some
$u\in U_\omega$ and $q\in Q_R$. Then we may write
\[
b_{\omega,u,q}-b_V
=
a\omega_\tau+y,
\qquad
|a|\lesssim Rs^2,
\quad
|y|\lesssim Rs,
\]
where $y\perp\omega_\tau$. Since $\angle(\omega,\omega_\tau)\lesssim s$, we obtain
\[
\left|
(b_{\omega,u,q}-b_V)\cdot\omega
\right|
\lesssim Rs^2.
\]
By \eqref{eq:b-properties}, it follows that
\[
|b_V\cdot\omega|
\lesssim
R^\beta+Rs^2
\lesssim
R^{1/2}+Rs^2.
\]

Let $E=\mathbb R^l\times\{0\}\subset\mathbb R^n$, and let $P_E$ be the orthogonal projection onto $E$. If $s$ is smaller
than a sufficiently small fixed constant, then the compatibility of $b_{\omega,u,q}$ with $V$
and \eqref{eq:b-properties} imply $|P_Eb_V|\sim R$.
Since $\omega\in S^{l-1}\subset E$, the preceding estimate shows that
the directions $\omega$ compatible with $V$ lie in an
$O(R^{-1/2}+s^2)$-neighborhood of a great sphere of codimension one,
inside the $s$-cap. Since the points of $\Omega_R$ are
$CR^{-1/2}$-separated, we obtain
\begin{equation}
\label{eq:omega-in-tau}
\#\left\{
\omega\in\Omega_R:
\begin{array}{l}
\theta_\omega\subset\tau,\\
\text{$\omega$ is compatible with $V$}
\end{array}
\right\}
\lesssim
R^{\frac{l-1}{2}}
s^{l-2}
\left(R^{-1/2}+s^2\right).
\end{equation}
If $s$ is bounded below by a fixed constant, the same
estimate follows from the trivial bound
\[
\#\{\omega\in\Omega_R:\theta_\omega\subset\tau\}
\lesssim
(R^{1/2}s)^{l-1}.
\]

Fix such an $\omega$. The points $\{cR\,u: u\in U_\omega\}$ are $R^{1/2}$-separated and lie on a sphere of radius comparable to $R$. If the corresponding centers are compatible with the same box $V$,
then the parameters $u$ must lie in an $O(s)$-cap of $S^{l-2}$.
Hence
\[
\#\left\{
u\in U_\omega:
\text{the corresponding center is compatible with }V
\right\}
\lesssim
(R^{1/2}s)^{l-2}.
\]

Finally, changing $q$ moves the center essentially in the short
direction of $V$. Since the short side of $V$ has length $Rs^2$,
\begin{equation}
\label{eq:q-in-V}
\#\left\{
q\in Q_R:
b_{\omega,u,q}\text{ is compatible with }V
\right\}
\lesssim
\min\{R^\beta,Rs^2\}.
\end{equation}

Combining \eqref{eq:omega-in-tau} and \eqref{eq:q-in-V} with the
count of the parameters $u$, we obtain
\[
\#\left\{
T\in\mathbb T_R:
\theta(T)\subset\tau,\;
T^\flat\subset V
\right\}
\lesssim
R^{l-\frac32}
s^{2l-4}
\left(R^{-1/2}+s^2\right)
\min\{R^\beta,Rs^2\}.
\]

Since every wave packet of the family has $L^2$ norm comparable to $1$, the
definition of the wave packet density and \eqref{eq:V-volume} yield
\begin{equation}
\label{eq:W-general}
W(f_R,B_R^{n+1})^2
\lesssim
\sup_{R^{-1/2}\leq s\leq1}
R^{l-n-\frac32}
s^{2l-n-5}
\left(R^{-1/2}+s^2\right)
\min\{R^\beta,Rs^2\}.
\end{equation}

We evaluate the supremum separately for odd and for even $n$.

\medskip

\noindent
\emph{Odd $n$.}
Here
\[
l=\frac{n+3}{2},
\qquad
\beta=\frac12.
\]
Thus \eqref{eq:W-general} gives
\[
W(f_R,B_R^{n+1})^2
\lesssim
\sup_{R^{-1/2}\leq s\leq1}
R^{-\frac n2}
s^{-2}
\left(R^{-1/2}+s^2\right)
\min\{R^{1/2},Rs^2\}.
\]
If $R^{-1/2}\leq s\leq R^{-1/4}$, then
\[
R^{-1/2}+s^2\lesssim R^{-1/2},
\qquad
\min\{R^{1/2},Rs^2\}=Rs^2,
\]
and hence
\[
R^{-\frac n2}
s^{-2}
\left(R^{-1/2}+s^2\right)
\min\{R^{1/2},Rs^2\}
\lesssim
R^{-\frac{n-1}{2}}.
\]
If $R^{-1/4}\leq s\leq1$, then
\[
R^{-1/2}+s^2\lesssim s^2,
\qquad
\min\{R^{1/2},Rs^2\}=R^{1/2},
\]
and again
\[
R^{-\frac n2}
s^{-2}
\left(R^{-1/2}+s^2\right)
\min\{R^{1/2},Rs^2\}
\lesssim
R^{-\frac{n-1}{2}}.
\]
Therefore,
\[
W(f_R,B_R^{n+1})^2
\lesssim
R^{-\frac{n-1}{2}}.
\]

\medskip

\noindent
\emph{Even $n$.}
Here
\[
l=\frac n2+2,
\qquad
\beta=0.
\]
Since $Rs^2\geq1$ for $s\geq R^{-1/2}$, we have $\min\{1,Rs^2\}=1$. Therefore, \eqref{eq:W-general} gives
\[
\begin{aligned}
W(f_R,B_R^{n+1})^2
&\lesssim
R^{-\frac{n-1}{2}}
\sup_{R^{-1/2}\leq s\leq1}
s^{-1}\left(R^{-1/2}+s^2\right) \lesssim
R^{-\frac{n-1}{2}}.
\end{aligned}
\]

Thus, for both parities of $n$,
\begin{equation}
\label{eq:W-proof}
W(f_R,B_R^{n+1})^2
\lesssim
R^{-\frac{n-1}{2}}.
\end{equation}
Up to constants, this is the wave packet density of a single wave packet of $L^2$ norm $1$, so the density of $f_R$ is as small as it can be.

\noindent
\textbf{Step 3. A lower bound for the $L^p$ norm.}

Let
\[
U_R
=
N_{CR^{1/2}}(\mathcal H_R)
\cap B_R^{n+1},
\]
restricted, if necessary, to a fixed smooth portion of the hyperboloid
on which all geometric constants are uniform.

Since the part of $\mathcal H_R$ in $B_R^{n+1}$ is an $l$-dimensional surface whose $l$-dimensional measure is comparable to $R^l$,
\begin{equation}
\label{eq:UR-volume}
|U_R|
\sim
R^l(R^{1/2})^{n+1-l}
=
R^{\frac{n+l+1}{2}}.
\end{equation}

Each plank of $\mathbb T_R$ lies in $U_R$ during a time interval of length comparable
to $R$. Consequently,
\[
\|e^{it\sqrt{-\Delta}}f_{\omega,u,q}\|_{L^2(U_R)}^2
\gtrsim R.
\]

We now choose the signs $\varepsilon_{\omega,u,q}$ independently and uniformly at random in $\{-1,1\}$. Since the cross terms have expectation zero,
\[
\mathbb E_\varepsilon
\|e^{it\sqrt{-\Delta}}f_R\|_{L^2(U_R)}^2
=
\sum_{\omega,u,q}
\|e^{it\sqrt{-\Delta}}f_{\omega,u,q}\|_{L^2(U_R)}^2
\gtrsim
R\,\#\mathbb T_R.
\]
Hence there is a choice of signs for which
\[
\|e^{it\sqrt{-\Delta}}f_R\|_{L^2(U_R)}
\gtrsim
R^{\frac12}
R^{\frac12(l-\frac32+\beta)}
=
R^{\frac l2-\frac14+\frac\beta2}.
\]
We fix these signs.

For every $p\geq2$, H\"older's inequality and
\eqref{eq:UR-volume} imply
\[
\begin{aligned}
\|e^{it\sqrt{-\Delta}}f_R\|_{L^p(B_R^{n+1})}
&\geq
\|e^{it\sqrt{-\Delta}}f_R\|_{L^p(U_R)}
\geq
|U_R|^{\frac1p-\frac12}
\|e^{it\sqrt{-\Delta}}f_R\|_{L^2(U_R)}
\\
&\gtrsim
R^{\frac{n+l+1}{2p}
+\frac{l-n-2}{4}
+\frac\beta2}.
\end{aligned}
\]

For odd $n$,
\[
l=\frac{n+3}{2},
\qquad
\beta=\frac12,
\]
and therefore
\begin{equation}
\label{lower-bound-0}
\|e^{it\sqrt{-\Delta}}f_R\|_{L^p(B_R^{n+1})}
\gtrsim
R^{\frac{3n+5}{4p}-\frac{n-1}{8}}.
\end{equation}

For even $n$,
\[
l=\frac n2+2,
\qquad
\beta=0,
\]
and therefore
\begin{equation}
\label{lower-bound-1}
\|e^{it\sqrt{-\Delta}}f_R\|_{L^p(B_R^{n+1})}
\gtrsim
R^{\frac{3n+6}{4p}-\frac n8}.
\end{equation}

\noindent
\textbf{Step 4. Conclusion.}

By \eqref{eq:F-unified} and \eqref{eq:W-proof}, for every $p\geq2$ and for both parities of $n$,
\[
R^{(n-1)\left(\frac12-\frac1p\right)}
\|f_R\|_2^{\frac{2}{p}}
W(f_R,B_R^{n+1})^{1-\frac{2}{p}}
\lesssim
R^{\frac{n-1}{4}+\frac1p}.
\]

If $n$ is odd, we compare this with \eqref{lower-bound-0}. The inequality of Theorem~\ref{thm:counterexample} holds with
\[
c_{n,p}
=
\frac{3n+5}{4p}-\frac{n-1}{8}
-\frac{n-1}{4}-\frac1p
=
\frac{3n+1}{4p}-\frac{3(n-1)}{8},
\]
and $c_{n,p}>0$ if and only if
\[
p
<
2+\frac8{3n-3}.
\]

If $n$ is even, we compare with \eqref{lower-bound-1}. The inequality of Theorem~\ref{thm:counterexample} holds with
\[
c_{n,p}
=
\frac{3n+6}{4p}-\frac n8
-\frac{n-1}{4}-\frac1p
=
\frac{3n+2}{4p}-\frac{3n-2}{8},
\]
and $c_{n,p}>0$ if and only if
\[
p
<
2+\frac8{3n-2}.
\]
In both cases the condition is $p<p_n^{+}$. This completes the proof of Theorem~\ref{thm:counterexample}.

\subsection{Proof of Theorem~\ref{thm:counter-example-multiscale-bound}}
We use the construction of the preceding proof with
$n=3$, so that $l=3$ and $\beta=\frac12$. Recall that the wave packets have the centers
\[
b_{\omega,u,q}=cRu+q\omega,
\qquad
\omega\in\Omega_R,\quad u\in U_\omega,\quad q\in Q_R,
\]
where $\Omega_R\subset S^2$ and
$U_\omega\subset S^2\cap\omega^\perp$ are
$CR^{-1/2}$-separated, while $Q_R$ is a $C$-separated subset of
$[-c_1R^{1/2},c_1R^{1/2}]$. Each wave packet is normalized so that
$\|f_{\omega,u,q}\|_2\sim1$. By \eqref{eq:F-unified} and \eqref{lower-bound-0},
\[
\|f_R\|_2\sim R,
\qquad
\|e^{it\sqrt{-\Delta}}f_R\|_{L^p(B_R^4)}
\gtrsim R^{\frac{7}{2p}-\frac14},
\qquad p\ge2.
\]
We shall prove that
\begin{equation}\label{eq:multiscale-density-upper-3d}
\widetilde{W}(f_R,B_R^4)^2
\lesssim R^{-3/4}.
\end{equation}
These three bounds give the inequality of Theorem~\ref{thm:counter-example-multiscale-bound} with
\[
c_p=\frac{7}{2p}-\frac14-\Bigl(1-\frac2p\Bigr)-\frac2p+\frac38\Bigl(1-\frac2p\Bigr)
=\frac{11}{4p}-\frac78,
\]
which is positive for $p<22/7$. It therefore remains to prove \eqref{eq:multiscale-density-upper-3d}.

Fix dyadic scales $s_1,s_2\in[R^{-1/2},1]$; by symmetry, we may assume that $s_2\le s_1$. Let $\tau$ be an $(s_1,s_2)$-cap, and let $V\parallel V_{\tau,R}$ be a corresponding test box. Then
\[
|V|\sim (Rs_1^2)(Rs_1)(Rs_2)=R^3s_1^3s_2.
\]
Write
\[
N(\tau,V)
:=
\#\bigl\{
T\in\mathbb T_R:
\theta(T)\subset\tau,\ T^\flat\subset V
\bigr\}.
\]

Since the directions in $\Omega_R$ are $CR^{-1/2}$-separated,
the number of directions that contribute to $N(\tau,V)$ is at most
\begin{equation} \label{caps-bound}
    \#\{\omega\in\Omega_R:\theta_\omega\subset\tau\}
\lesssim Rs_1s_2.
\end{equation}

Fix one such $\omega$, and let $P_{\omega^\perp}$ denote the orthogonal projection onto $\omega^\perp$. Since $P_{\omega^\perp}b_{\omega,u,q}=cRu$, the condition $b_{\omega,u,q}\in V$ implies $cRu\in P_{\omega^\perp}V$.
The latter set has diameter $O(Rs_1)$. As the points $cRu$ are
$CR^{1/2}$-separated on a circle of radius $cR$, the number of
possible $u$ is therefore at most $O(R^{1/2}s_1)$.

For each fixed pair $(\omega,u)$, there are trivially
$O(R^{1/2})$ possible values of $q$. If $s_1\le R^{-1/4}$, we can do better. Let $e_\tau$ denote the unit vector along the edge of
$V$ of length $Rs_1^2$. Since this edge points in the radial direction of $\tau$, we have $|\omega\cdot e_\tau|\gtrsim1$. Projecting $b_{\omega,u,q}=cRu+q\omega$ onto $e_\tau$, we see that the
condition $b_{\omega,u,q}\in V$ confines $q$ to an interval of
length $O(Rs_1^2)$. Since the parameters $q$ are $C$-separated
and $Rs_1^2\ge1$, we obtain, for all admissible $s_1$,
\begin{equation} \label{q-bound}
\#\{q\in Q_R:b_{\omega,u,q}\in V\}
\lesssim \min\{R^{1/2},Rs_1^2\}.
\end{equation}

Combining \eqref{caps-bound} and \eqref{q-bound} with the
count of the parameters $u$ yields
\[
N(\tau,V)
\lesssim
R^{3/2}s_1^2s_2
\min\{R^{1/2},Rs_1^2\}.
\]
Since each wave packet of the family has squared $L^2$ norm comparable
to $1$, it follows that
\[
\begin{aligned}
\widetilde{W}(f_R,B_R^4)^2
&\lesssim
\sup_{R^{-1/2}\le s_2\le s_1\le1}
\frac{
R^{3/2}s_1^2s_2\min\{R^{1/2},Rs_1^2\}
}{
R^3s_1^3s_2
}\\
&=
\sup_{R^{-1/2}\le s_1\le1}
R^{-3/2}s_1^{-1}
\min\{R^{1/2},Rs_1^2\}.
\end{aligned}
\]
The expression in the last supremum equals
\[
\begin{cases}
R^{-1/2}s_1,
& R^{-1/2}\le s_1\le R^{-1/4},\\[3pt]
R^{-1}s_1^{-1},
& R^{-1/4}\le s_1\le1.
\end{cases}
\]
Both expressions are bounded by $R^{-3/4}$, which proves
\eqref{eq:multiscale-density-upper-3d}.

\end{document}